\documentclass[12pt, amscd]{amsart}
\usepackage{amscd}
\usepackage{verbatim}
\usepackage[dvips]{color}
\usepackage{color}

\input xy
\xyoption{all}

\usepackage{amssymb}

\DeclareMathAlphabet{\cat}{OT1}{cmss}{m}{sl}

\newtheorem{theorem}{Theorem}[section]

\newtheorem{proposition}[theorem]{Proposition}
\newtheorem{lemma}[theorem]{Lemma}
\newtheorem{corollary}[theorem]{Corollary}

\theoremstyle{definition}
\newtheorem{remark}[theorem]{Remark}

\newtheorem{example}[theorem]{Example}

\newcommand{\tens}{\otimes}

\newcommand{\gmu}{\boldsymbol{\mu}}

\newcommand{\ad}{\operatorname{ad}}

\newcommand{\ch}{\operatorname{char}}

\newcommand{\disc}{\operatorname{disc}}

\newcommand{\Br}{\operatorname{Br}}

\newcommand{\Sym}{\operatorname{Sym}}

\newcommand{\gPGL}{\operatorname{\mathbf{PGL}}}

\newcommand{\gO}{\operatorname{\mathbf{O}}}
\newcommand{\gSp}{\operatorname{\mathbf{Sp}}}

\newcommand{\gm}{\operatorname{\mathbb{G}}_m}

\newcommand{\gPGU}{\operatorname{\mathbf{PGU}}}
\newcommand{\gPGO}{\operatorname{\mathbf{PGO}}}
\newcommand{\gPGSp}{\operatorname{\mathbf{PGSp}}}

\newcommand{\Image}{\operatorname{Im}}
\newcommand{\identity}{\operatorname{Id}}

\newcommand{\ed}{\operatorname{ed}}

\newcommand{\QH}{\operatorname{QH}}

\newcommand{\A}{\mathbb{A}}

\newcommand{\cT}{\mathcal T}

\newcommand{\cS}{\mathcal S}

\usepackage[hypertex]{hyperref}

\title[Essential dimension of projective orthogonal, symplectic groups] % colontitle
{Essential dimension of projective orthogonal and symplectic groups of small degree}

\author
[S. Baek] {Sanghoon Baek}

\address
{Department of Mathematics and Statistics, University of Ottawa, 585 King
Edward, Ottawa, ON K1N6N5, Canada}

\email {sbaek@uottawa.ca}

\thanks{The work has been partially supported from the Fields Institute and from Zainoulline's NSERC Discovery grant 385795-2010.}

\begin{document}
\begin{abstract}
In this paper, we study the essential dimension of classes of central simple algebras with involutions of index less or equal to $4$. Using structural theorems for simple algebras with involutions, we obtain the essential dimension of projective and symplectic groups of small degree.

\end{abstract}
\maketitle

%\tableofcontents

\section{Introduction}

Let $F$ be a field, $A$ a central simple $F$-algebra, and $(\sigma,f)$ a quadratic pair on $A$ (see \cite[5.B]{Book}). A morphism of algebras with quadratic pair $\phi:(A,\sigma,f)\to (A',\sigma',f')$ is an $F$-algebra morphism $\phi:A\to A'$ such that $\sigma'\circ \phi=\phi\circ \sigma$ and $f\circ \phi=f'$. For any field extension $K/F$, we write $(A,\sigma,f)_{K}$ for $(A\tens_{F}K, \sigma\tens \identity_{K}, f_{K})$, where $f_{K}:\Sym(A_{K},\sigma_{K})\to K$.

For $n\geq 2$, let $D_{n}$ denote the category of central simple $F$-algebras of degree $2n$ with quadratic pair, where the morphisms are the $F$-algebra homomorphisms which preserve the quadratic pairs and let $A_{1}^{2}$ denote the category of quaternion algebras over an \'etale quadratic extension of $F$, where the morphisms are the $F$-algebra isomorphism. Then, there is an equivalence of groupoids:
\begin{equation}\label{DA}
D_{2}\equiv A_{1}^{2};
\end{equation}
see \cite[Theorem 15.7]{Book}.

Moreover, if we consider the full subgroupoid $^{1}\!A_{1}^{2}$ of $A_{1}^{2}$ whose objects are $F$-algebras of the form $Q\times Q'$, where $Q$ and $Q'$ are quaternion algebras over $F$, and the full subgroupoid $^{1}D_{n}$ of $D_{n}$ whose objects are central simple algebras over $F$ with quadratic pair of trivial discriminant, then the equivalence in (\ref{DA}) specialize to the following equivalence of subgroupoids:
\begin{equation}\label{DATrivial}
^{1}D_{2}\equiv ^{1}\!\!\!A_{1}^{2};
\end{equation}
see \cite[Corollary 15.12]{Book}.

For $n\geq 1$, we denote by $C_{n}$ be the category of central simple $F$-algebras of degree $2n$ with symplectic involution, where the morphisms are the $F$-algebra isomorphism which preserve the involutions.

By Galois cohomology, there are canonical bijections (see \cite[\S 29.D,F]{Book})
\begin{equation}\label{DAGalois}
D_{n}\longleftrightarrow H^{1}(F,\gPGO_{2n})
\end{equation}
and
\begin{equation}\label{CAGalois}
C_{n}\longleftrightarrow H^{1}(F,\gPGSp_{2n}).
\end{equation}

Let $\cT:\cat{Fields}/F\to\cat{Sets}$ be a functor from the category $\cat{Fields}/F$ of field extensions over $F$ to the category $\cat{Sets}$ of sets and let $p$ be a prime. We denote by $\ed(\cT)$ and $\ed_{p}(\cT)$ the essential dimension and essential p-dimension of $\cT$, respectively. We refer to \cite[Def. 1.2]{BerhuyFavi03} and \cite[Sec.1]{Merkurjev09} for their definitions. Let $G$ be an algebraic group over $F$. The \emph{essential dimension $\ed(G)$} (respectively, \emph{essential $p$-dimension $\ed_{p}(G)$}) of $G$ is defined to be $\ed(H^1(-,G))$ (respectively, $\ed_{p}(H^1(-,G))$), where $H^1(E,G)$ is the nonabelian cohomology set with respect to the finitely generated faithfully flat topology (equivalently, the set of isomorphism classes of $G$-torsors) over a field extension $E$ of $F$.

A morphism $\cS\to \cT$ from $\cat{Fields}/F$ to $\cat{Sets}$ is called \emph{$p$-surjective} if for any $E\in \cat{Fields}/F$ and any $\alpha\in \cT(E)$, there is a finite field extension $L/E$ of degree prime to $p$ such that $\alpha_{L}\in\Image(\cS(L)\to \cT(L))$. A morphism of functors $\cS\to \cT$ from $\cat{Fields}/F$ to $\cat{Sets}$ is called \emph{surjective} if for any $E\in \cat{Fields}/F$, $\cS(E)\to \cT(E)$ is surjective. Obviously, any surjective morphism is $p$-surjective for any prime $p$. Such a surjective morphism gives an upper bound for the essential ($p$)-dimension of $\cT$ and a lower bound for the essential ($p$)-dimension of $\cS$,
\begin{equation}\label{ctcssssssssss}
\ed(\cT)\leq \ed(\cS) \text{ and } \ed_{p}(\cT)\leq \ed_{p}(\cS);
\end{equation}
see \cite[Lemma 1.9]{BerhuyFavi03} and \cite[Proposition 1.3]{Merkurjev09}.

\begin{example}\label{pgsp2}
Let $(M_{2}(F),\gamma)\in C_{1}$, where $\gamma$ is the canonical involution on $M_{2}(F)$. As $(M_{2}(K),\gamma_{K})\simeq (M_{2}(F),\gamma)\tens_{F}K$ for any field extension $K/F$, we have $\ed((M_{2}(F),\gamma))=0$.

Assume that $\ch(F)\neq 2$. The exact sequence \[1\to \gmu_{2}\to \gSp_{2}\to \gPGSp_{2}\to 1\] induces the connecting morphism $\partial:H^{1}(-,\gPGSp_{2})\to \Br_{2}(-)$ which sends a pair $(Q,\gamma)$ of a quaternion algebra with canonical involution to the Brauer class $[Q]$. As this morphism is nontrivial, by \cite[Corollary 3.6]{BerhuyFavi03} we have $\ed(\gPGSp_{2})\geq 2$ (or by Lemma \ref{forgetfulpgo}). Consider the morphism $\gm^{2}\to C_{1}$ defined by $(x,y)\mapsto ((x,y),\gamma)$, where $\gamma$ is the canonical involution. As this morphism is surjective, by (\ref{ctcssssssssss}) we have $\ed(C_{1})\leq 2$, thus $\ed(\gPGSp_{2})=2$. This can be recovered from the exceptional isomorphism $\gPGSp_{2}\simeq \gO^{+}_{3}$.
\end{example}

\emph{Acknowledgements}: I am grateful to A.~Merkurjev for useful discussions. I am also grateful to J.~P. Tignol and S. Garibaldi for helpful comments.

\section{Essential dimension of projective orthogonal and symplectic groups associated to central simple algebras of index $\leq 4$}

First, we compute upper bounds for the essential dimension of certain classes of simple algebras with involutions of index less or equal to $2$.

Let $(Q,\gamma)$ be a pair of a quaternion over a field $F$ and the canonical involution. For any field extension $K/F$ and any integer $n\geq 3$, we write $\QH_{n}^{+}(K)$ (respectively, $\QH_{n}^{-}(K)$) for the set of isomorphism classes of $(M_n(Q),\sigma_{h})$, where $\sigma_{h}$ is the adjoint involution on $M_n(Q)$ with respect to a hermitian form (respectively, skew-hermitian form) $h$ (with respect to $\gamma$). If $n$ is odd, we write $^{1}\!\QH_{n}^{-}(K)$ for the set of isomorphism classes of $(M_n(Q),\sigma_{h})$, where $\sigma_{h}$ is the adjoint involution on $M_n(Q)$ with respect to a skew-hermitian form $h$ with $\disc(\sigma_{h})=1$.

\begin{lemma}\label{quaternionichermitian}
Let $F$ be a field and $n\geq 3$ any integer. Then
\begin{enumerate}
\item $\ed(\QH_{n}^{+})\leq n+1$.
\smallskip
\item $\ed(\QH_{n}^{-})\leq 3n-3$ if $\ch(F)\neq 2$.
\smallskip
\item $\ed($$^{1}\!\QH_{n}^{-})\leq 3n-4$ if $\ch(F)\neq 2$.
\end{enumerate}
\end{lemma}
\begin{proof}
$(1)$ If $h$ is a hermitian form, then $h=\langle t_{1},t_{2},\cdots, t_{n}\rangle$ for some $t_{i}\in F$. We consider the affine variety
\[X=
\begin{cases}
\gm^{2}\times \A_{F}^{n-1} &\text{ if $\ch(F)\neq 2$},\\
\gm\times \A_{F}^{n} &\text{ if $\ch(F)= 2$},\\
\end{cases}
\]
and define a morphism $X(K)\to \QH_{n}^{+}(K)$ by
\[
(a,b,t_{1},\cdots,t_{n-1})\mapsto
\begin{cases}
((a,b)\tens M_{n}(K), \sigma_{\langle 1,t_{1},\cdots, t_{n-1}\rangle}) &\text{ if $\ch(F)\neq 2$},\\
([a,b)\tens M_{n}(K), \sigma_{\langle 1,t_{1},\cdots, t_{n-1}\rangle}),&\text{ if $\ch(F)= 2$}.\\
\end{cases}
\]
As a scalar multiplication does not change the adjoint involution, this morphism is surjective. Therefore, by (\ref{ctcssssssssss}), we have $\ed(\QH_{n}^{+})\leq n+1$.

$(2)$ From now we assume that $\ch(F)\neq 2$. If $h$ is a skew-hermitian form, then $h=\langle q_{1},q_{2},\cdots, q_{n}\rangle$ for some pure quaternions $q_{i}\in Q$. We consider the affine variety $Y=\gm^{2}\times \A^{1}\times \A^{3(n-2)}$ with coordinates $(a,b,c,t_{1},\cdots, t_{3n-6})$ and the conditions $ac^{2}+b\neq 0$, $at_{1+3(k-1)}^{2}+bt_{2+3(k-1)}^{2}-abt_{3k}^{2}\neq0$ for all $1\leq k\leq n-2$. Define a morphism $\phi_{K}: Y(K)\to \QH_{n}^{-}(K)$ by
\[(a,b,c,t_{1},\cdots,t_{3n-6})\mapsto ((a,b)\tens M_{n}(K), \sigma_h),\]
where $p=i$, $q=ci+j$, $r_{k}=t_{1+3(k-1)}i+t_{2+3(k-1)}j+t_{3k}ij$ for $1\leq k\leq n-2$, and $h=\langle p,q,r_{1},\cdots, r_{3n-2}\rangle$.

We show that $Y$ is a classifying variety for $\QH_{n}^{-}$. Suppose that we are given a quaternion $K$-algebra $Q=(a,b)$ and a skew hermitian form $h=\langle p,q, r_{1},\cdots, r_{n-2}\rangle$ for some pure quaternions $p,q,r_{k}$. We can find a scalar $c\in K$ such that $p$ and $q-cp$ anticommute, thus we have $Q\simeq (p^{2}, (q-cp)^{2})$. For $1\leq k\leq n-2$, let
\begin{equation}\label{rkkk}
r_{k}=t_{1+3(k-1)}p+t_{2+3(k-1)}(q-cp)+t_{3k}p(q-cp)
\end{equation}
with $t_{1},t_{2},\cdots, t_{3n-6}\in K$. Then $(M_{n}(Q),\sigma_{h})\simeq (M_{n}((p^{2}, (q-cp)^{2})),\sigma_{h})$ is the image of $\phi_{K}$. Therefore, by (\ref{ctcssssssssss}), we have $\ed(\QH_{n}^{-})\leq 3n-3$.

\medskip

$(3)$ Assume that $n=2m+1$ for $m\geq 1$. We consider the variety $Y$ in $(2)$ with an additional condition
\[-a(ac^{2}+b)\prod_{k=1}^{m-1}at_{1+3(k-1)}^{2}+bt_{2+3(k-1)}^{2}-abt_{3k}^{2}=\prod_{k=m}^{n-2}at_{1+3(k-1)}^{2}+bt_{2+3(k-1)}^{2}-abt_{3k}^{2}.\]
We show that this variety with the same morphism $\phi_{K}$ in $(2)$ is a classifying variety for $^{1}\!\QH_{n}^{-}$. Suppose that we are given a quaternion $K$-algebra $Q=(a,b)$ and a skew hermitian form $h=\langle p,q, r_{1},\cdots, r_{n-2}\rangle$ for some pure quaternions $p,q,r_{k}$. We do the same procedure as in $(1)$, so that we have $(M_{n}(Q),\sigma_{h})\simeq (M_{n}((p^{2}, (q-cp)^{2})),\sigma_{h})$ and (\ref{rkkk}). As $\disc(\sigma_{h})=1$, there is a scalar $d\in K^{\times}$ such that
\[-p^{2}q^{2}r_{1}^{2}\cdots r_{m-1}^{2}=(\frac{d}{r_{m}^{2}\cdots r_{n-2}^{2}})^{2}r_{m}^{2}\cdots r_{n-2}^{2}.\]
We set $f=d/{r_{m}^{2}\cdots r_{n-2}^{2}}$. As a scalar multiplication does not change the adjoint involution, we can modify $h$ by the scalar $f$. As $(p^{2}, (q-cp)^{2})\simeq (f^{2}p^{2}, f^{2}(q-cp)^{2})$, $(M_{n}(Q),\sigma_{h})$ is the image of $\phi_{K}$. Therefore, by (\ref{ctcssssssssss}), we have $\ed($$^{1}\!\QH_{n}^{-})\leq 3n-4$.

\end{proof}
\begin{remark}
The main idea of the proof of the case where $h$ is a skew-hermitian form is from Merkurjev's work on algebras of degree $4$ in his private note.
\end{remark}

Assume that $n$ is odd. Then we have \[\QH_{n}^{+}=C_{n}, \QH_{n}^{-}=D_{n}, \text{ and } ^{1}\!\QH_{n}^{-}=^{1}\!D_{n}.\] Hence, by \cite[Theorem 4.2]{Book} and the exceptional isomorphism $\gPGO_{6}\simeq \gPGU_{4}$, we have
\begin{corollary}
Assume that $n\geq 3$ is odd. Then
\begin{enumerate}
\item $\ed(\gPGSp_{2n})\leq n+1$.
\item $\ed(\gPGO_{2n})\leq 3n-3$ if $\ch(F)\neq 2$. In particular, $\ed(\gPGU_{4})\leq 6$.
\item $\ed(\gPGO^{+}_{2n})\leq 3n-4$ if $\ch(F)\neq 2$.
\end{enumerate}
\end{corollary}

\begin{remark}
In fact, $\ed_{2}(\gPGSp_{2n})=\ed(\gPGSp_{2n})=n+1$ for $n\geq 3$ odd and $\ch(F)\neq 2$. The lower bound was obtained by Chernousov and Serre in \cite[Theorem 1]{CS} and the exact value was obtained by Macdonald in \cite[Proposition 5.1]{Mac}.
\end{remark}

\begin{lemma}\label{lempgl2}\cite[Section 2.6]{BM09}
Let $F$ be a field of characteristic different from $2$. Then
\[\ed_{2}(\gPGL_{2}^{\times n})=\ed(\gPGL_{2}^{\times n})=2n.\]
\end{lemma}
\begin{proof}
By \cite[Lemma 1.11]{BerhuyFavi03}, we have
\[\ed_{2}(\gPGL_{2}^{\times n})\leq \ed(\gPGL_{2}^{\times n})\leq n\cdot \ed(\gPGL_{2})=2n.\]
On the other hand, the natural morphism
\begin{equation}\label{surdec}
H^{1}(-,\gPGL_{2}^{\times n})\to \cat{Dec}_{2^{n}}(-)
\end{equation}
is surjective, where $\cat{Dec}_{2^{n}}(K)$ is the set of all decomposable algebras of degree $2^{n}$ over a field extension $K/F$, hence, by (\ref{ctcssssssssss}), we have
\[2n=\ed_{2}(\cat{Dec}_{2^{n}})\leq \ed_{2}(\gPGL_{2}^{\times n}).\]

\end{proof}

\begin{lemma}\label{forgetfulpgo}
Let $F$ be a field of characteristic different from $2$. Then
\[\ed_{2}(\gPGO_{2^r}), \ed_{2}(\gPGSp_{2^r})\geq
\begin{cases}
2 &\text{ if }r=1,\\
4 &\text{ if }r=2,\\
(r-1)2^{r-1} &\text{ if }r\geq 3.\\
\end{cases}
\]
\end{lemma}
\begin{proof}
Consider the forgetful functors
\begin{equation}\label{forge}
H^{1}(-,\gPGO_{2^r})\to \cat{Alg}_{2^r,2}
\end{equation}
and
\begin{equation}\label{forge2}
H^{1}(-,\gPGSp_{2^r})\to \cat{Alg}_{2^r,2},
\end{equation}
where $\cat{Alg}_{2^r,2}(K)$ is the set of isomorphism classes of simple algebras of degree $2^{n}$ and exponent dividing $2$ over a field extension $K/F$. These functors are surjective by a theorem of Albert.

It is well known that $\ed_{2}(\cat{Alg}_{2,2})=2$, $\ed_{2}(\cat{Alg}_{4,2})=4$. For $r\geq 3$, we have $\ed_{2}(\cat{Alg}_{2^r,2})\geq (r-1)2^{r-1}$  by \cite[Theorem]{BM10}. Therefore, by (\ref{ctcssssssssss}), we have the above lower bound for $\ed_{2}(\gPGO_{2^r})$ and $\ed_{2}(\gPGSp_{2^r})$.
\end{proof}

\smallskip

The following Lemma \ref{symplecticdecomp}(1) was proved by Rowen in \cite[Theorem B]{Rowcsa} (see also \cite[Proposition 16.16]{Book}) and Lemma \ref{symplecticdecomp}(2) was proved by Serhir and Tignol in \cite[Proposition]{STignol}. We shall need the explicit forms of involutions on the decomposed quaternions as in (1):

\begin{lemma}\label{symplecticdecomp}
Let $F$ be a field of characteristic different from $2$. Let $(A,\sigma)$ be a central simple $F$-algebra of degree $4$ with a symplectic involution $\sigma$.
\begin{enumerate}
\item If $A$ is a division algebra, then we have \[(A,\sigma)\simeq (Q,\sigma|_{Q})\tens (Q',\gamma),\]
where $\sigma|_{Q}$ is an orthogonal involution defined by $\sigma|_{Q}(x_{0}+x_{1}i+x_{2}j+x_{3}k)=x_{0}+x_{1}i+x_{2}j-x_{3}k$ with a quaternion basis $(1,i,j,k)$ for $Q$ and $\gamma$ is the canonical involution on a quaternion algebra $Q'$.
\item If $A$ is not a division algebra, then we have
\[(A,\sigma)\simeq (M_{2}(F),\ad_{q})\tens (Q',\gamma),\]
where $q$ is a $2$-dimensional quadratic form, $\ad_{q}$ is the adjoint involution on $M_{2}(F)$, and $\gamma$ is the canonical involution on a quaternion algebra $Q'$.
\end{enumerate}
\end{lemma}
\begin{proof}
(1) By \cite[Proposition 5.3]{Rowcsa}, we can choose a $i\in A\backslash F$ such that $\sigma(i)=i$ and $[F(i):F]=2$. Let $\phi$ be the nontrivial automorphism of $F(i)$ over $F$. By \cite[Proposition 5.4]{Rowcsa}, there is a $j\in A\backslash F$ such that $\sigma(j)=j$ and $ji=\phi(i)j$. Then $i$ and $j$ generate a quaternion algebra $Q$, $\sigma|_{Q}(i)=i$, $\sigma|_{Q}(j)=j$, and $\sigma|_{Q}(k)=-k$ with $k=ij$. Hence, $\sigma|_{Q}$ is an orthogonal involution on $Q$. By the double centralizer theorem, we have $A\simeq Q\tens C_{A}(Q)$, where $C_{A}(Q)$ is the centralizer of $Q\subset A$ and is isomorphic to quaternion algebra $Q'$ over $F$. By \cite[Proposition 2.23]{Book}, the restriction of $\sigma$ on $Q'$ is the canonical involution $\gamma$.

(2) See \cite[Proposition]{STignol}.
\end{proof}

\smallskip

\begin{proposition}\label{oopo}
Let $F$ be a field of characteristic different from $2$. Then
\begin{enumerate}
\item $\ed_{2}(\gPGO^{+}_{4})=\ed(\gPGO^{+}_{4})=4$.
\smallskip
\item $\ed_{2}(\gPGO_{4})=\ed(\gPGO_{4})=4$.
\smallskip
\item $\ed_{2}(\gPGSp_{4})=\ed(\gPGSp_{4})=4$.
\end{enumerate}
\end{proposition}

\begin{proof}
(1) By the exceptional isomorphism (\ref{DATrivial}), we have
\[\gPGO^{+}_{4}=\gPGL_{2}\times \gPGL_{2}.\]
The proof follows from Lemma \ref{lempgl2} with $n=2$.

\smallskip

(2) By Lemma \ref{forgetfulpgo}, we have $\ed_{2}(\gPGO_{4})\geq 4$. For the opposite inequality, we consider the affine variety $X$ defined in $\A^{4}_{F}$ with the coordinates $(a,b,c,e)$ by $e(a^{2}-b^{2}e)(c^{2}-e)\neq 0$. Define a morphism $X\to \!A_{1}^{2}$ by
\[(a,b,c,e)\mapsto \begin{cases}
(a+b\sqrt{e}, c+\sqrt{e}) & \text{if $F(\sqrt{e})$ is a quadratic field extension},\\
(a,b)\times (c,\sqrt{e}) & \text{otherwise},
\end{cases}
\]

We show that $X$ is a classifying variety for $\!A_{1}^{2}$. Let $Q=(a+b\sqrt{e},c+d\sqrt{e})$ be a quaternion algebra over a quadratic extension $L=F(\sqrt{e})$. If $b=d=0$, we can modify $c$ by a norm of $L(\sqrt{a})/L$, hence we may assume that $d\neq 0$. Similarly, we can assume that $d=1$, replacing $e$ by $ed^{2}$. Thus, the morphism $X\to \!A_{1}^{2}$ is surjective. By (\ref{ctcssssssssss}), $\ed(\!A_{1}^{2})\leq 4$. Hence, the opposite inequality $\ed(\gPGO_{4})\leq 4$ comes from the exceptional isomorphism (\ref{DA}) and the canonical bijection (\ref{DAGalois}).

\smallskip

(3) By Lemma \ref{forgetfulpgo}, we have $\ed_{2}(\gPGSp_{4})\geq 4$. For the opposite inequality, we define a morphism $\gm^{4}\to C_{2}$ by
\[(x,y,z,w)\mapsto \begin{cases}
((x,y),\sigma)\tens ((z,w),\gamma) & \text{if $x\neq 1$},\\
(M_{2}(F),\ad_{q})\tens ((z,w),\gamma) & \text{if $x=1$},
\end{cases}
\]
where $\sigma$ is an involution defined by $\sigma(x_{0}+x_{1}i+x_{2}j+x_{3}k)=x_{0}+x_{1}i+x_{2}j-x_{3}k$ with a quaternion basis $(1,i,j,k)$ of the quaternion algebra $(x,y)$, $q=\langle 1,y\rangle$ is a quadratic form, and $\gamma$ is the canonical involution on the quaternion algebra $(z,w)$. Note that multiplying any two dimensional quadratic form by a scalar does not change the adjoint involution. By Lemma \ref{symplecticdecomp}, this morphism is surjective. Therefore, by (\ref{ctcssssssssss}), we have $\ed(C_{2})\leq 4$, hence the result follows from the canonical bijection (\ref{CAGalois}).

\end{proof}

\begin{remark}
\hfill
\begin{enumerate}
\item Assume that $F$ is a field of characteristic $2$. By \cite[Corollary 2.2]{Baek}, we have $\ed_{2}(\cat{Alg}_{4,2})=\ed_{2}(\cat{Dec}_{4})\geq 3$. As the morphisms (\ref{surdec}) and (\ref{forge}) are surjective, we get $\ed_{2}(\gPGO^{+}_{4})\geq 3$ and $\ed_{2}(\gPGO_{4})\geq 3$, respectively. On the other hand, the upper bounds in Proposition \ref{oopo} (1) and (2) still hold, hence $3\leq \ed(\gPGO^{+}_{4}),  \ed(\gPGO_{4})\leq 4$.
\item As $\ed(\gO^{+}_{5})=4$ by \cite[Theorem 10.3]{Re00}, Proposition \ref{oopo} (3) can be recovered from the exceptional isomorphism $\gPGSp_{4}\simeq \gO^{+}_{5}$.
\end{enumerate}
\end{remark}


\begin{thebibliography}{10}

\bibitem{Baek}
S. Baek, \emph{Essential dimension of simple algebras in positive characteristic}, C. R. Math. Acad. Sci. Paris \textbf{349} (2011) 375--378.

\bibitem{BM09}
S. Baek and A. Merkurjev, \emph{Invariants of simple algebras},
  Manuscripta Math. \textbf{129} (2009), no. 4, 409--421.

\bibitem{BM10}
S.~Baek and A.~Merkurjev, \emph{Essential dimension of central simple algebras}, to appear in Acta Math.

\bibitem{BerhuyFavi03}
G.~Berhuy and G.~Favi, \emph{Essential dimension: a functorial
  point of view $($after {A}. {M}erkurjev$)$}, Doc. Math. \textbf{8} (2003),
  279--330 (electronic).

\bibitem{Book}
M.-A.~Knus, A.~Merkurjev, M.~Rost, and J.-P. Tignol, \emph{The book
of involutions}, American Mathematical Society, Providence, RI, 1998, With a
  preface in French by J.\ Tits.

\bibitem{CS}
V.~Chernousov and J.-P.~Serre, \emph{Lower bounds for essential dimensions via orthogonal representations},
  J. Algebra \textbf{305} (2006), no. 2, 1055--1070.

\bibitem{Mac}
M.~L.~MacDonald, \emph{Cohomological invariants of odd degree {J}ordan algebras},
  Math. Proc. Cambridge Philos. Soc. \textbf{145} (2008), no. 2, 295--303.


\bibitem{Merkurjev09}
A.~S. Merkurjev, \emph{Essential dimension}, Quadratic
forms---algebra,
  arithmetic, and geometry, Contemp. Math., vol. 493, Amer. Math. Soc.,
  Providence, RI, (2009), pp.~299--325.

\bibitem{Re00}
Z.~Reichstein, \emph{On the notion of essential dimension for algebraic groups}, Transform. Groups 5 (2000), no. 3, 265–-304.

\bibitem{Rowcsa}
L.~Rowen, \emph{Central simple algebras}, Israel J. Math.
\textbf{29}
  (1978), no.~2-3, 285--301.

\bibitem{STignol}
A.~Serhir and J.-P. Tignol, \emph{The discriminant of a decomposable symplectic involution}, J. Algebra \textbf{273} (2004), no. 2, 601--607.

\end{thebibliography}
\end{document}